\documentclass[11pt]{amsart} \textwidth=14.5cm \oddsidemargin=1cm
\usepackage{amsmath,amsthm,amssymb, hyperref}
\usepackage[all]{xy}
\usepackage{tikz}

\newtheorem{thm}{Theorem} [section]

\newtheorem{lem}[thm]{Lemma}
\newtheorem{prop}[thm]{Proposition}

\theoremstyle{definition}
\newtheorem{definition}[thm]{Definition}
\newtheorem{example}[thm]{Example}

\theoremstyle{remark}
\newtheorem{rem}[thm]{Remark}

\numberwithin{equation}{section}

\usepackage{amsmath,amsthm,amssymb}
\usepackage[all]{xy}
\usepackage{tikz}
\usepackage{verbatim}
\usepackage{amsmath}
\usepackage{amsxtra}
\usepackage{amscd}
\usepackage{amsthm}
\usepackage{amsfonts}
\usepackage{amssymb}
\usepackage{eucal}
\usetikzlibrary{arrows}

\usepackage{latexsym,todonotes}

\newcommand{\nc}{\newcommand}
 \nc{\Z}{{\mathbb Z}}
 \nc{\C}{{\mathbb C}}
 \nc{\N}{{\mathbb N}}
 \nc{\F}{{\mf F}}
 \nc{\Q}{\mathbb{Q}}
 \nc{\la}{\lambda}
 \nc{\ep}{\epsilon}\nc{\End}{\text{End}}
 \nc{\h}{\mathfrak h}
 \nc{\mc}{\mathcal}
 \nc{\LL}{\texttt{L}}
 \nc{\HBm}{\mathcal{H}_{m}}
 \nc{\inv}{^{-1}}
 \nc{\qu}{\quad}
 \nc{\WBm}{W_{B_m}}
 \nc{\U}{\bold{U}}
 \nc{\Uj}{\U^{\jmath}}
 \nc{\Ui}{\U^{\imath}}
 \nc{\ibar}{\psi_{\imath}}
 \nc{\jbar}{\psi_{\jmath}}
 \nc{\I}{\mathbb{I}}
 \nc{\be}{e}
 \nc{\bff}{f}
 \nc{\bk}{k}
 \nc{\bt}{t}
 \nc{\VV}{\mathbb V}
 \nc{\Ii}{I}
 \nc{\hf}{\frac{1}{2}}
 \nc{\Ij}{I}
 \nc{\hgt}{\text{ht}}

\title[Multiparameter quantum Schur duality of type B]{Multiparameter quantum Schur duality of type B}
 \author[Bao]{Huanchen Bao}
\address{Department of Mathematics, University of Maryland, College Park, MD 20742}
\email{huanchen@math.umd.edu}

\author[Wang]{Weiqiang Wang}
\address{Department of Mathematics, University of Virginia, Charlottesville, VA 22904}
\email{ww9c@virginia.edu}

\author[Watanabe]{Hideya Watanabe}
\address{Department of Mathematics, Tokyo Institute of Technology, 2-12-1 Oh-okayama, Meguro-ku, Tokyo 152-8551, Japan}
\email{watanabe.h.at@m.titech.ac.jp}
\subjclass[2010]{Primary~17B10}
\keywords{}
\date{}

\begin{document}
\maketitle

\begin{abstract}
We establish a Schur type duality between a coideal subalgebra of the quantum group of type A 
and the Hecke algebra  of type B with 2 parameters. 
We identify the $\imath$-canonical basis on the tensor product of the natural representation with 
Lusztig's  canonical basis of the type B Hecke algebra with unequal parameters associated to a weight function. 
\end{abstract}

\section{Introduction}

\subsection{}
Jimbo (\cite{J86}) established a quantum Schur duality between the quantum group $\U$ of type A
and the Hecke algebra $\mathcal{H}_{A_{m-1}}$, which asserts that their actions on the tensor product $\VV^{\otimes m}$
of the natural representation $\VV$ of $\U$ commute and form double centralizers. 
To facilitate further discussions, we take the base field to be $\Q(p,q)$ for two parameters $p,q$ instead of $\Q(q)$. 

Let $\HBm$ be the Hecke algebra of type $B_m$ with two parameters $p,q$, which contains
$\mathcal{H}_{A_{m-1}}$ as a subalgebra and admits one extra generator $H_0$; see \eqref{eq:HB} for the definition. 
The Hecke algebra $\HBm$ acts naturally on $\VV^{\otimes m}$ as well, where $H_0$ acts on the first tensor factor only. 
On the other hand, there is a notion of quantum symmetric pair, $(\U, \Ui)$, where $\Ui$ is a coideal subalgebra of $\U$.
The algebra $\Ui$ allows for some freedom of choices of parameters; see Letzter \cite{Le99} (also see \cite{BK15}). We make a particular choice of the parameters for $\Ui$ in this paper depending on $p$ and $q$. The coideal subalgebra $\Ui$ acts on $\VV^{\otimes m}$ naturally.

Our first main result ($\imath$-Schur duality) asserts that the actions of $\Ui$ and $\HBm$ on $\VV^{\otimes m}$ commute and form double centralizers. 
This double centralizer theorem was established by the first two authors in \cite{BW13}
 in the specialization when $p=q$,
and then by the first author in \cite{B16} in the specialization when $p=1$ and $q$ is generic. 
The multiparameter $(\Ui, \HBm)$-duality in this paper is a natural generalization
and synthesis of these earlier cases.

\subsection{}

Let us return to the setting of the type A Schur-Jimbo duality for a moment.
The space $\VV^{\otimes m}$ admits a (parabolic) Kazhdan-Lusztig basis via
its identification as a direct sum of permutation modules for the Hecke algebra $\mathcal{H}_{A_{m-1}}$.  
By the work of Lusztig \cite{L94}, there exists a
canonical basis on the tensor product $\U$-module $\VV^{\otimes m}$.
It is well known that the canonical basis and the Kazhdan-Lusztig basis on $\VV^{\otimes m}$ coincide 
(cf. Frenkel-Khovanov-Kirillov \cite{FKK}).

Let $\LL: W_{B_m} \rightarrow \Z$ be a weight function on the Weyl group of type B; cf. \cite{L03}. 
Via a homomorphism $v^\LL : \Q(p,q) \rightarrow \Q(v)$ with an indeterminate $v$,
we consider the specializations $\HBm^\LL$, $\U_\LL$, $\Ui_\LL$ and $\VV^{\otimes m}_\LL$ over $\Q(v)$. 
The multiparameter $(\Ui, \HBm)$-duality leads to the $(\Ui_\LL, \HBm^\LL)$-duality on $\VV^{\otimes m}_\LL$ under such a specialization.

Lusztig \cite{L03} constructed a distinguished bar invariant basis of $\HBm^\LL$ (called an $\LL$-basis in this paper), 
which specializes to the Kazhdan-Lusztig ($\texttt{KL}$) basis (\cite{KL79}) when $\LL$ is the length function $\ell $ of a Weyl group. It is straightforward to adapt
Lusztig's construction to the parabolic setting; cf. \cite{Deo87}. 
Thus $\VV^{\otimes m}_\LL$ admits an $\LL$-basis through its identification as a direct sum of permutation modules over $\HBm^\LL$.

An $\imath$-canonical basis on a tensor product $\U$-module (for example, $\VV^{\otimes m}$) when $p=q$
was constructed in \cite{BW13},  which is invariant with respect to a new bar involution introduced therein.
Moreover,  the $\imath$-canonical basis on $\VV^{\otimes m}$ when $p=q$ 
(which corresponds to the case when $\LL =\ell$) is identified with the type B $\texttt{KL}$-basis.
An easy modification of the construction {\em loc. cit.}  leads to an $\imath$-canonical basis on $\VV^{\otimes m}_\LL$, for any weight function $\LL$. Note that the $\imath$-canonical basis on $\VV^{\otimes m}_\LL$ depends on $\LL$, since the algebra $\Ui_\LL$ depends on $\LL$. 
We refer to \cite{BW16} for a general theory of $\imath$-canonical bases for  quantum symmetric pairs with parameters.

The second main result of this paper is that the $\imath$-canonical basis on $\VV^{\otimes m}_\LL$ 
coincides with the $\LL$-basis on $\VV^{\otimes m}_\LL$, for any weight function $\LL$. 
For another distinguished choice of $\LL$ (which corresponds to taking $p=1$), the $\LL$-basis (which is also the $\imath$-canonical basis)
is the \texttt{KL}-basis of type D; see \cite{B16}. 

\subsection{}

The constructions and proofs in this paper are mostly adapted from \cite{BW13}, some of which has been known to us for some time. 
Nevertheless, the new setting does require various new nontrivial 2-parameter formulas,
and thus we present explicitly the precise details which are new in our setting. 
The detailed constructions in this paper (where $\U$ is the quantum group for $\mathfrak{sl}_k$) 
depend much on the parity of $k$, so we treat the two cases separately. 
In Sections~\ref{sec:Ui} and \ref{sec:CBLB}, we treat the case when $k=2r+2$ is even. 
We establish in Section~\ref{sec:Ui} the $(\Ui, \HBm)$-duality on $\VV^{\otimes m}$.
In Section~\ref{sec:CBLB}, we study the $\imath$-canonical basis  on $\VV^{\otimes m}_\LL$
associated to a weight function $\LL$, and show it coincides with the $\LL$-basis.
In Section~\ref{sec:Uj} we present the analogous constructions when $k=2r+1$ is odd.  
\\

\noindent {\bf Acknowledgements.} HB is partially supported by an AMS-Simons travel grant,
and he thanks Max Planck Institute
for Mathematics for support which facilitated this collaboration. WW is partially supported by the NSF grant DMS-1405131, 
and he thanks George Lusztig for helpful discussions on \cite{L03}.

\section{The $\imath$-Schur duality with 2 parameters}
 \label{sec:Ui}
 
In this section, we establish a Schur-type duality between a coideal subalgebra $\Ui$ 
of the quantum group for $\mathfrak{sl}_{2r+2}$
and the Hecke algebra of type $B$ in 2 parameters.  

\subsection{The quantum symmetric pair $(\U,\Ui)$}\label{Ui}

Let $r\ge 0$ be an integer. We set
\[
\I= \{-r, -r+1, \dots, r\},
\qquad
\I^\imath = \{1, \dots, r\}.
\]
Let $\Pi= \big \{\alpha_i=\varepsilon_{i-\hf}-\varepsilon_{i+\hf} \mid i \in \I \big \}$
be the simple system of type $A_{2r+1}$, and $\Phi$ the associated root system. Denote the weight lattice by 
$$
\Lambda = \sum_{{i \in \I} } \big( \Z \varepsilon_{i - \hf} + \Z \varepsilon_{i + \hf} \big). 
$$  

Let $p$ and $q$ be indeterminates. 
Let $\U_q(\mathfrak{sl}_{2r+2})$ denote the quantum group of type $A_{2r + 1}$ over $\Q(q)$ 
with the standard generators $E_{i}$, $F_{i}$ and $K^{\pm 1}_{i}$ for $i \in \I$ (see, e.g., \cite[\S1.2]{BW13} for a precise definition). 
Let 
\[\U = \U_q(\mathfrak{sl}_{2r+2}) \otimes_{\Q(q)} \Q(p,q).
\]
We denote by $\psi$ the bar involution on $\U$ (more conventionally denoted by $\bar{\phantom{x}}$ ), that is, a $\Q$-algebra involution of $\U$ such that 
\begin{equation}  \label{eq:psi}
\psi(q) =q^{-1},\quad \psi(p) =p^{-1},  \quad \psi(E_{i}) = E_{i},  \quad\psi(F_{i}) = F_{i},  \quad \psi(K_{i})= K_{i}^{-1}.
\end{equation}
We shall use the comultiplication $\Delta: \U \rightarrow \U \otimes \U$ as follows:
\begin{align}
\Delta(E_{i}) = 1 \otimes E_{i} + E_{i} \otimes K_{i}\inv, \;
 \Delta(F_{i}) = F_{i} \otimes 1 + K_{i} \otimes F_{i}, \;
  \Delta(K_{i}) = K_{i} \otimes K_{i}. 
 \label{eq:Delta}
\end{align}

We review  the quantum symmetric pair $(\U,\Ui)$ over $\Q(p,q)$  \cite{Le99} 
with the following Satake diagram:

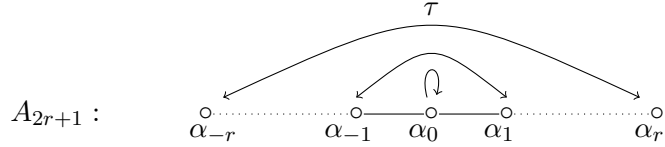
\begin{center}
\begin{tikzpicture}
\draw (-2,0) node {$A_{2r+1}:$};
 \draw[dotted]  (0.1,0) node[below] {$\alpha_{-r}$} -- (1.9,0) node[below] {$\alpha_{-1}$} ;
 \draw (2.1,0) -- (2.9,0) node[below]  {$\alpha_{0}$};
 \draw (3.1,0)-- (3.9,0) node[below] {$\alpha_{1}$};
 \draw[dotted] (4.1,0) -- (5.9,0) node[below] {$\alpha_{r}$} ;
\draw (0,0) node (-r) {$\circ$};
 \draw (2,0) node (-1) {$\circ$};
\draw (3,0) node (0) {$\circ$};
\draw (4,0) node (1) {$\circ$}; 
\draw (6,0) node (r) {$\circ$};
\draw[<->] (-r.north east) .. controls (3,1.5) .. node[above] {$\tau$} (r.north west) ;
\draw[<->] (-1.north) .. controls (3,1) ..  (1.north) ;
\draw[<->] (0) edge[<->, loop above] (0);
\end{tikzpicture}
\end{center}
Let $\Ui$ be the $\Q(p,q)$-subalgebra of $\U$ generated by (for $i \in \{1, \dots, r\}$):
 \begin{align}
  \label{eq:embedi}
 \begin{split}
&\bk_{i} = K_{i}K^{-1}_{{-i}}, \quad   \bt = E_{0} +qF_{0}K^{-1}_{0} + \frac{p-p^{-1}}{q-q^{-1}} K_{0}^{-1},  \\
&\be_{i} =  E_{i} + F_{{-i}}K^{-1}_{i}, \quad
\bff_{i} = E_{{-i}} + K^{-1}_{{-i}} F_{i}. 
\end{split}
\end{align}
Note that $\Ui$ is a right coideal subalgebra of $\U$, that is, we have $\Delta(\Ui) \subset \Ui \otimes \U$.

The algebra $\Ui$ has a presentation as an algebra over $\Q(p,q)$ generated by $\be_{i}$, $\bff_{i}$, $\bk_{i}^{\pm 1}  \ (i \in \{1, \ldots, r\})$, and $t$, 
subject to the following relations for $i, j \in \{ 1, \ldots, r \}$ (see \cite{BW13, BK15}):
\begin{align*}
 \bk_{i} \bk_{i}^{-1} &= \bk_{i}^{-1} \bk_{i} =1, \quad
 \bk_{i} \bk_{j} = \bk_{j} \bk_{i}, \displaybreak[0]\\
 \bk_{i} \be_{j} \bk_{i}^{-1} &= q^{(\alpha_i-\alpha_{-i}, \alpha_j)} \be_{j}, 
 \quad
 \bk_{i} \bff_{j} \bk_{i}^{-1}  = q^{-(\alpha_i-\alpha_{-i}, \alpha_j)} \bff_{j},  
 \quad
 \bk_{i}\bt\bk^{-1}_{i} = \bt, \displaybreak[0]\\
 \be_{i} \bff_{ j} -\bff_{ j} \be_{i} &= \delta_{i,j} \frac{\bk_{i} -\bk^{-1}_{i}}{q-q^{-1}}, \displaybreak[0]\\
 \be_{i} \be_{ j} &= \be_{ j} \be_{i}, \quad
  \bff_{i} \bff_{ j}  = \bff_{ j}  \bff_{i},   \qquad\qquad\ \ \ \ \qquad &\text{if }& |i-j|>1, \displaybreak[0]\\
 \be_{i}\bt &=\bt\be_{i}, \quad
  \bff_{i}\bt =\bt\bff_{i}, \qquad\quad &\text{if }&  i > 1, \displaybreak[0]\\
  \be_{i}^2 \be_{ j} +\be_{ j} \be_{i}^2 &= (q+q^{-1}) \be_{i} \be_{ j} \be_{i},  \quad
 \bff_{i}^2 \bff_{ j} +\bff_{ j} \bff_{i}^2 = (q+q^{-1}) \bff_{i} \bff_{ j} \bff_{i},
    \ \ \quad\quad &\text{if }& |i-j|=1,\displaybreak[0]\\
 \be_{ 1}^2\bt + \bt\be_{ 1}^2 &= (q+q^{-1}) \be_{ 1}\bt\be_{ 1},\quad
 \bff_{ 1}^2\bt + \bt\bff_{ 1}^2 = (q+q^{-1}) \bff_{ 1}\bt\bff_{1},\displaybreak[0]\\
 \bt^2\be_{ 1} + \be_{ 1}\bt^2 &= (q + q^{-1}) \bt\be_{ 1}\bt + \be_{ 1},\quad
 \bt^2\bff_{ 1} + \bff_{ 1}\bt^2 = (q + q^{-1}) \bt\bff_{ 1}\bt + \bff_{ 1}.\displaybreak[0]
\end{align*}

The next lemma follows by inspection using the above presentation of $\Ui$.

\begin{lem}\label{lem:bar}
There exists a unique $\Q$-algebra bar involution $\ibar$ on $\Ui$ such that 
$p \mapsto p\inv$, $q \mapsto q\inv$, $\bk_{i} \mapsto \bk_{i}\inv$, $\be_{i} \mapsto \be_{i}$, $\bff_{i} \mapsto \bff_{i}$, and $\bt \mapsto \bt$, for $i \in \I^\imath$.
\end{lem}

\begin{rem}
The presentation of the algebra $\Ui$ here is the same as the algebra in almost the same notation in \cite[\S2.1]{BW13}. 
The parameter $\kappa \in \Q(p,q)$ in the embedding (see \eqref{eq:embedi}), 
$
\bt = E_{ 0} +qF_{ 0}K^{-1}_{ 0} + \kappa  K_{ 0}^{-1},
$
is irrelevant to the presentation of the algebra $\Ui$. This phenomenon was first observed in \cite{Le99}. 
\end{rem}

Let $\widehat{\U}$ be the completion of the $\Q (p,q)$-vector space $\U$ 
with respect to the following descending sequence of subspaces 
$\U^+ \U^0 \big(\sum_{\hgt(\mu) \geq N}\U_{\mu}^- \big)$,   for $N \ge 1.$
Then we have the obvious embedding of $\U$ into $\widehat{\U}$. 
We let $\widehat{\U}^-$ be the closure of $\U^-$ in $\widehat{\U}$, and so $\widehat{\U}^- \subseteq \widehat{\U}$. 
By continuity the $\Q(p,q)$-algebra structure 
on $\U$ extends to a $\Q(p,q)$-algebra structure on $ \widehat{\U}$.  The bar involution $\psi$  on $\U$ extends 
by continuity to an {\em anti-linear} (i.e., $p\mapsto p^{-1}, q\mapsto q^{-1}$) involution on $\widehat{\U}$, also denoted by  $\psi$.
Denote by $\N$ the set of nonnegative integers.
\begin{prop}\cite[Theorem~2.10]{BW13}
  \label{prop:intetwiner}
There is a unique family of elements $\Upsilon_\mu \in \U^-_{-\mu}$ for $\mu \in \N\Pi$ such that $\Upsilon_0 = 1$, and $\Upsilon = \sum_{\mu} \Upsilon_\mu \in \widehat{\U}^-$ intertwines the bar involution $\psi$ on $\U$ and the bar involution $\psi_{\imath}$ on $\Ui$; that is, $\Upsilon$ satisfies the following identity:
\begin{align}
\psi_\imath(u) \cdot \Upsilon = \Upsilon \cdot \psi({u}),\qquad \text{ for all } u \in \Ui. 
\nonumber
\end{align}
\end{prop}
Note that \cite[Theorem~2.10]{BW13} was proved in the setting of $p=q$ in \eqref{eq:embedi}, but the same proof works here.
We shall call $\Upsilon$ the intertwiner. 

Consider a $\Q(q)$-valued function $\zeta$ on $\Lambda$ (which  is independent of the parameter $p$), such that $\text{for all } \mu \in \Lambda, \; i \in \I^\imath$, we have 
\begin{align}
\label{eq:zeta}
\begin{split}\zeta (\mu+\alpha_0)&=-q \zeta (\mu),  \\ 
\zeta (\mu+\alpha_i) &= -q^{(\alpha_i -\alpha_{-i}, \mu+\alpha_i)} \zeta (\mu), \\
\zeta (\mu+\alpha_{-i}) &= -q^{(\alpha_{-i}, \mu+\alpha_{-i}) - (\alpha_{i}, \mu)} \zeta (\mu).  
\end{split}
\end{align}

Such $\zeta$ clearly exists. For any weight $\U$-module $M$, 
we define a $\Q(p,q)$-linear map on $M$ associated with $\zeta$ as follows:
\begin{align}
\label{eq:zetatd}
\begin{split}
\widetilde{\zeta}&: M \longrightarrow M, 
 \\
\widetilde{\zeta}  (m &)  =  \zeta (\mu)m, \quad \text{for all } m \in M_{\mu}.
\end{split}  
\end{align}

Let $w_0$ denotes the longest element of the Weyl group $W_{A_{2r+1}}$ of type $A_{2r+1}$ and $T_{w_0}$ the associated braid group element. 
The following proposition is a 2-parameter variant of {\cite[Theorem 2.18]{BW13}}, with the same proof.

\begin{prop}\label{prop:mcT}
For any finite-dimensional $\U$-module $M$, the composition map
\[
\mc{T} := \Upsilon\circ \widetilde{\zeta} \circ T_{w_0}: M \longrightarrow M
\] 
is a $\Ui$-module isomorphism.
\end{prop}

%
%

\subsection{Hecke algebra of type $B$ with 2 parameters}

We often write $q_i =q$ for $1\le i \le m-1$ and $q_0=p$. 
Let $\HBm$ be the Hecke algebra of type $B_m$ with 2 parameters over $\Q(p,q)$ 
generated by $H_0, H_1, \ldots, H_{m-1}$ 
and subject to the following relations: 
\begin{equation}\label{eq:HB}
\begin{split}
 (H_i-q_i^{-1})(H_i+q_i)&=0, \qu \text{for }i \ge 0, \\
H_i H_{i+1} H_i &= H_{i+1} H_i H_{i+1}, \qu \text{for } i \geq 1,  \\
H_0 H_1 H_0 H_1 = H_1 H_0 H_1 H_0, \quad H_i H_j &= H_j H_i, \qu \text{for } |i-j| > 1.
\end{split}
\end{equation}

Let $\WBm$ be the Weyl group of type $B_m$ with simple reflections $s_0, s_1, \ldots, s_{m-1}$. 
For each $w \in \WBm$ with a reduced expression $w = s_{i_1} \cdots s_{i_r}$, the products $H_w =H_{i_1} \cdots H_{i_r}$ 
and $q_w =q_{i_1} \cdots q_{i_r}$ are independent of the choice of the reduced expressions.

\subsection{The $(\Ui,\HBm)$-duality}
Let $\Ii = \I \pm \hf$. Let $\VV = \bigoplus_{a \in \Ii} \Q(p,q)u_a$ be the natural representation of $\U$. The $\U$-module structure of $\VV$ can be visualized as follows:
\begin{align}
\xymatrix{
u_{r+\frac{1}{2}} \ar@/_/[r]_{E_{r}} & 
u_{r-\frac{1}{2}} \ar@/_/[l]_{F_{r}} \ar@/_/[r]_{E_{{r-1}}} & 
\cdots \ar@/_/[l]_{F_{{r-1}}} \ar@/_/[r]_{E_{{1}}} & 
u_{\frac{1}{2}} \ar@/_/[l]_{F_{{1}}} \ar@/_/[r]_{E_{{0}}} & 
u_{-\frac{1}{2}} \ar@/_/[l]_{F_{{0}}} \ar@/_/[r]_{E_{{-1}}} & 
\cdots \ar@/_/[l]_{F_{{-1}}} \ar@/_/[r]_{E_{{-r}}} & 
u_{-r-\frac{1}{2}} \ar@/_/[l]_{F_{{-r}}}}. \nonumber
\end{align}

We denote by $\VV^{\otimes m}$ the $m$-th tensor product of $\VV$, which is naturally an $\U$-module via iterated comultiplication. 
Hence $\VV^{\otimes m}$ is a $\Ui$-module by restriction.

For any $f =(f(1),\ldots,f(m)) \in \Ii^m$, we define
\begin{equation}  \label{Mf}
M_f= u_{f(1)} \otimes \cdots \otimes u_{f(m)}.
\end{equation}
The Weyl group $W_{B_m}$ acts on $\Ii^m$ on the right in the obvious way: for $j \geq 1$ and $i\in I$,
\begin{align}
(f \cdot s_j)(i) &= \begin{cases}
f(j+1) \qu & \text{if } i = j; \\
f(j) \qu & \text{if } i = j+1; \\
f(i) \qu & \text{otherwise};
\end{cases} \nonumber\\
(f \cdot s_0)(i) &= \begin{cases}
-f(1) \qu & \text{if } i = 1; \\
f(i) \qu & \text{otherwise}.
\end{cases} \nonumber
\end{align}
The Hecke algebra $\HBm$ acts naturally on $\VV^{\otimes m}$ on the right as follows:
\begin{align}  \label{Haction} 
\begin{split}
 M_f \cdot H_i &=
 \begin{cases}
 q^{-1}M_f, & \text{ if }  f(i) = f(i+1);\\
 M_{f \cdot  s_i}, & \text{ if }  f(i) < f(i+1);\\
 M_{f \cdot  s_i} + (q^{-1} - q) M_{f}, & \text{ if } f(i) > f(i+1);
 \end{cases}
 \\
 M_f \cdot H_0 &= 
 \begin{cases}
 M_{f \cdot s_0}, &\text{ if } f(1) > 0;\\
  M_{f \cdot  s_0} + (p^{-1} - p) M_{f}, & \text{ if } f(1) <0.
 \end{cases}
 \end{split}
\end{align}
We shall depict the actions of $\Ui$ and $\HBm$ on $\VV^{\otimes m}$ as
\begin{equation}  \label{LRact}
\Ui \; \stackrel{\Phi}{\curvearrowright} \;
\VV^{\otimes m} \; \stackrel{\Psi}{\curvearrowleft}
\; \HBm.
\end{equation}

Now we fix $\zeta$ in \eqref{eq:zeta} by setting $\zeta(\epsilon_{-r-\frac{1}{2}}) = 1$. Then, we have
\begin{align}
\zeta(\epsilon_{-i - \frac{1}{2}}) = (-q)^{-r+i},  \qquad \forall i \in \{ r, r-1, \ldots, -r-1 \}.
\nonumber
\end{align} 

\begin{lem}\label{lem:H0}
The actions of $H_0$ and $\mc{T}^{-1}$ on $\VV$ coincides. 
\end{lem}

\begin{proof}
Let 
\begin{align*}
\VV^+ &= \bigoplus_{j\in  \{0,\dots, r\}} \Q(p,q)  \big( u_{-j-\frac{1}{2}} + pu_{j+\frac{1}{2}} \big), 
\\
\VV^- &= \bigoplus_{j\in  \{0,\dots, r\}} \Q(p,q) \big(u_{-j-\frac{1}{2}} - p\inv u_{j+\frac{1}{2}} \big). 
\end{align*}
By direct calculations, we have   (for $i \in \I^\imath$, $j\in \{0,\dots, r\}$)
\begin{align}
\bt \cdot \big(u_{-j-\frac{1}{2}} + pu_{j+\frac{1}{2}}\big) &= \frac{pq^{\delta_{j,0}}-p^{-1}q^{-\delta_{j,0}}}{q-q^{-1}}\big(u_{-j-\frac{1}{2}} + pu_{j+\frac{1}{2}}\big), \nonumber\\
\bff_{i} \cdot \big(u_{-j-\frac{1}{2}} + pu_{j+\frac{1}{2}}\big) &= \delta_{i, j+1} \cdot \big(u_{-(j+1)-\frac{1}{2}} + p u_{(j+1) + \frac{1}{2}}\big), \nonumber\\
\be_{i} \cdot \big(u_{-j-\frac{1}{2}} + pu_{j+\frac{1}{2}}\big) &= \delta_{i, j} \cdot \big(u_{-(j-1)-\frac{1}{2}} + p u_{(j-1) + \frac{1}{2}}\big), \nonumber\\
%
\bt \cdot  \big(u_{-j-\frac{1}{2}} - p\inv u_{j+\frac{1}{2}}\big) &=  \frac{pq^{-\delta_{j,0}}-p^{-1}q^{\delta_{j,0}}}{q-q^{-1}}\big(u_{-j-\frac{1}{2}} - p\inv u_{j+\frac{1}{2}}\big), \nonumber\\
\bff_{i} \cdot \big(u_{-j-\frac{1}{2}} - p\inv u_{j+\frac{1}{2}}\big) &= \delta_{i, j+1} \cdot \big(u_{-(j+1)-\frac{1}{2}} - p\inv u_{(j+1) + \frac{1}{2}}\big), \nonumber\\
\be_{i} \cdot \big(u_{-j-\frac{1}{2}} - p\inv u_{j+\frac{1}{2}}\big) &= \delta_{i, j} \cdot \big(u_{-(j-1)-\frac{1}{2}} - p\inv u_{(j-1) + \frac{1}{2}}\big). \nonumber
\end{align}
Hence, we have $\VV =\VV^+ \oplus \VV^-$ as a $\Ui$-module. 

Since we have
$
T_{w_0}(u_{j + \frac{1}{2}}) = (-q)^{r-j}  u_{-j-\frac{1}{2}}, \nonumber
$
we obtain $\widetilde{\zeta} \circ T_{w_0}(u_{j+\frac{1}{2}}) = u_{-j-\frac{1}{2}}$. 
On the other hand, one computes the first terms of $\Upsilon$ as $\Upsilon_{\alpha_0} = (p - p\inv)F_{\alpha_0}$, which is the only term with non-trivial action for weight reason. Hence,  we have that
$
\mc{T}(u_{-\frac{1}{2}}) = u_{\frac{1}{2}}, \mc{T}(u_{\frac{1}{2}}) = u_{-\frac{1}{2}} - (p\inv - p)u_{\frac{1}{2}},  
$  
which implies the actions of $\mc{T}^{-1}$ on $\VV^+$ and $\VV^-$ are given  by scalar multiplication by $p^{-1}$ and $-p$, respectively. The
lemma is proved.
\end{proof}

Now, we state our first main theorem.

\begin{thm}  [$\imath$-Schur duality]
\label{thm:UiHB}
The actions of $\Ui$ and $\HBm$ on $\VV^{\otimes m}$ \eqref{LRact} commute and form double centralizers; that is, 
$$
\Phi(\Ui)=
\End_{\HBm}(\VV^{\otimes m}), \quad
\End_{\Ui}(\VV^{\otimes m})^{\operatorname{op}} =
\Psi(\HBm).
$$
\end{thm}

\begin{proof}
Recall $\Ui$ acts as a subalgebra of $\U$ on $\VV^{\otimes m}$, and the Hecke algebra $\mathcal H_{A_{m-1}}$ 
is a subalgebra of $\HBm$ generated by $H_i$, for $1\le i\le m-1$.
Hence it follows from the $q$-Schur duality of type $A$ (\cite{J86}) that the action of $H_i$, for $1\le i\le m-1$, 
commutes with the action of $\U$, and hence of $\Ui$.
Note that $H_0$ acts only on the first factor of $\VV^{\otimes m}$. On the other hand, the commutativity of the actions of $H_0$ and $\Ui$ follows by
Proposition~\ref{prop:mcT}, Lemma~\ref{lem:H0}, and that the $\Ui$ is a right coideal subalgebra of $\U$. 

The double centralizer property is the multiparameter version of {\cite[Theorem~ 5.4]{BW13}}, and it follows by a deformation argument in the same way as in {\em loc. cit.}.
\end{proof}

\section{The $\LL$-bases and $\imath$-canonical bases}
  \label{sec:CBLB}

\subsection{Quantum symmetrizers}

The type $B$ Hecke algebra $\HBm$ \eqref{eq:HB} has a unique $\Q$-algebra bar involution $\bar{}:\HBm \rightarrow \HBm$ such that
\begin{align}
\overline{q}_i = q_i\inv,\quad \overline{H_w} = H_{w\inv}\inv, \qu \text{for } i \ge 0, w \in \WBm. \nonumber
\end{align}

For any subset $J \subseteq \{ 0, 1, \ldots, m-1 \}$, let $W_J$ be the parabolic subgroup of $W$ generated by $\{ s_j \mid j \in J \}$. 
Let $\mc{H}_J$ be the $\Q(p,q)$-subalgebra of $\HBm$ generated by $\{ H_j \mid j \in J \}$, 
and ${}^J W$ be the set of minimal length coset representatives for $W_J \setminus W$. We define
a quantum symmetrizer for the algebra $\mc H_J$:
\begin{align} \label{eq:eta}
\eta_{_J} = q_{w_{_J}} \sum_{x \in W_J} q_x\inv H_x \in \mc{H}_J, 
\end{align}
where $w_{_J}$ denotes the longest element of $W_J$.

\begin{lem}\label{lem:parabolic}
Let $J \subseteq \{ 0, 1, \ldots, m-1 \}$. Then, the following hold:
\begin{enumerate}
\item $\eta_{_J} H_j = q_j\inv \eta_{_J}$, for all $j \in J$.

\item For $w \in {}^J W$ and $j \in J$, we have
\begin{align}
(\eta_{_J} H_w) \cdot H_j = \begin{cases}
q_j\inv \eta_{_J} H_w \qu & \text{if } ws_j \notin {}^J W; \\
\eta_{_J} H_{ws_j} \qu & \text{if } ws_j \in {}^J W \text{ and } w < ws_j; \\
\eta_{_J} H_{ws_j} + (q_j\inv - q_j) \eta_{_J} H_w \qu & \text{if } ws_j \in {}^J W \text{ and } ws_j < w.
\end{cases} \nonumber
\end{align}
\item $\overline{\eta_{_J}} = \eta_{_J}$.
\end{enumerate}
\end{lem}

\begin{proof}
Part~$(1)$ is proved by a direct calculation, and $(2)$ follows from $(1)$. 
Part~ (3)  can be (essentially)  found in \cite[\S12]{L03}.
\end{proof}

\subsection{The $\LL$-bases}

Recall \cite[\S3.1]{L03}
that a map $\LL : \WBm \longrightarrow \Z$ is called a {\em weight function} if it satisfies
\begin{align}
\LL(y  w) = \LL(y) + \LL(w) \nonumber
\end{align}
for all $y, w \in \WBm$ such that $\ell(y w) = \ell(y) + \ell(w)$.
(Such a weight function is determined by values $\LL(s_0)$ and $\LL(s_1)$.)

We fix a weight function $\LL$. If $m=1$, we adopt the convention that $\LL(s_1)=1$. Let $v$ be an indeterminate. 
We consider a $\Q$-algebra homomorphism 
\[
v^\LL : \Q(p,q) \longrightarrow \Q(v),
\qquad p \mapsto v^{\LL(s_0)}, \ q \mapsto v^{\LL(s_1)}.
\] 
We shall regard $\Q(v)$ as a $\Q(p,q)$-module via the $\Q$-algebra homomorphism $v^\LL$.
By a base change, we introduce the following algebras/spaces over $\Q(v)$:
\begin{equation}
 \HBm^{\LL} = \HBm \otimes_{\Q(p,q)} \Q(v),
\quad \Ui_{\LL}  = \Ui \otimes_{\Q(p,q)} \Q(v),
\quad \VV_\LL = \VV \otimes_{\Q(p,q)} \Q(v).
\end{equation}
We shall use the old notations of the generators of $\HBm$ for generators of $\HBm^\LL$ (and similarly for $\Ui_\LL$ and $\VV_\LL$). The bar involution on   $ \HBm$ (as well as on $\Ui$ and $\VV$) induces a bar involution  on $ \HBm^{\LL}$  (as well as on $\Ui_\LL$ and $\VV_\LL$) such that $\overline{v} = v^{-1}$.

The following is a straightforward variant of Lusztig \cite[Theorem 5.2]{L03} 
(who treats the regular representation, i.e., the $J=\emptyset$ case).

\begin{prop}  \label{KLB}
\cite{L03, Deo87} 
Let  $J \subseteq \{ 0, 1, \ldots, m-1 \}$. Then, for each $w \in {}^J W$, there exists a unique element $C^J_w \in \eta_{_J} \HBm^\LL$ such that
\begin{enumerate}
\item $\overline{C^J_w} = C^J_w$,
\item 
$
C^J_w \in  \eta_{_J} \big(H_w + \sum_{{y \in {}^J W, \, y < w}} v \Z[v] H_y \big). 
$
\end{enumerate}
Moreover, the elements $\{C^J_w~ \vert~ w \in {}^J W\}$ forms a $\Q(v)$-basis of $\eta_{_J} \HBm^\LL$ (called the $\LL$-basis).
\end{prop}

\begin{proof}
Thanks to Lemma~\ref{lem:parabolic}, the proof is the same as for the usual \texttt{KL}-setting {\cite[Propositions~ 3.1, 3.2]{Deo87}}. 
\end{proof}

\begin{rem}
We write $C^J_w =C^{\LL,J}_w$  to emphasize the dependence on the weight function $\LL$. 
By replacing $H_0$ and $H_i \ (i \ge 1)$ by $e_0 H_0$ and $e_1 H_i$, respectively, where $e_0, e_1 \in \{ 1, -1 \}$, 
one obtains an isomorphism $\HBm^\LL\cong \HBm^{\LL'}$, where $\LL'$ is the weight function determined by 
$\LL'(s_0) = e_0 \LL(s_0)$ and $\LL'(s_1) = e_1 \LL(s_1)$.
Moreover, one checks that the isomorphism is compatible with bar involutions, and
it sends the $\LL$-basis to the $\LL'$-basis up to sign, that is, $C^{\LL,J}_w \mapsto (-1)^{\ell(w)} C^{\LL',J}_w$. 
This observation is valid in the general setting of \cite{L03}. 
Therefore we may assume that a weight function is nonnegative integer valued if needed.
\end{rem}

\subsection{The $\imath$-canonical bases}
  \label{sec:iCB}

It is well known that there exists a bar-involution $\psi$ on the tensor product of several simple finite-dimensional $\U_\LL$-modules,
such as $\VV_\LL^{\otimes m}$, using the quasi-$R$-matrix $\Theta$ ({\cite[Chap.~ 4]{L94}}). 
Following \cite[Proposition~3.10]{BW13} we can define another anti-linear (i.e., $v\mapsto v^{-1}$) 
involution on $\VV_\LL^{\otimes m}$ as (recall the definition of $\Upsilon$ in Proposition~\ref{prop:intetwiner})
\begin{equation}
  \label{eq:ipsi}
\psi_\imath = \Upsilon \circ \psi. 
\end{equation}
By construction, $\psi_\imath$ is well defined on $\VV_\LL^{\otimes m}$ and fixes all $M_f$ such that $0 < f(1) \leq f(2) \leq \cdots \leq f(m)$.

\begin{rem}
We use the same notation $\psi$ (as in \cite{L94}) for both the anti-linear involution on $\U$ (as well as the specialization $\U_\LL$) and the anti-linear involution on the  $\U$-module $\VV^{\otimes m}$ (as well as the $\U_\LL$-module $\VV_\LL^{\otimes m}$), since they are compatible.
Similarly we use the same notation $\psi_\imath$ in a multiple of settings.
\end{rem}

To develop a theory of $\imath$-canonical basis, besides the new bar involution $\psi_\imath$, we also
need to establish the integrality of the intertwiner $\Upsilon$. 
The following is an $\LL$-variant of \cite[Theorem~4.18]{BW13}. 
\begin{prop}\label{prop:integralform}
Let $\mathcal{A} = \Z[v,v^{-1}]$ and ${}_{\mathcal{A}}\U$ be the $\mathcal{A}$-form of $\U$. Then, we have $\Upsilon_\mu \in {}_{\mathcal{A}}\U$ for any $\mu \in \N[\mathbb{I}]$.
\end{prop} 

\begin{proof}
Following the strategy of the proof of \cite[Theorem~4.18]{BW13}, 
the proof of the integrality of $\Upsilon$ is reduced 
to verifying that the intertwiner is integral for the case $\mathbb{I} = \{0\}$ (which is the counterpart of \cite[Lemma~4.6]{BW13}). 

We write $\Upsilon_c = \Upsilon_{c\alpha_0} = \gamma_c E^{(c)}$ for $c \ge 0$. Note that $\gamma_0=1$ by definition. The same computation as \cite[Lemma~4.6]{BW13} shows that 
\begin{align*}
\gamma_{c+1} &= - \big(v^{\LL(s_1)}-v^{-\LL(s_1)} \big) v^{-c \LL(s_1)} 
\Big(v^{\LL(s_1)} [c]_{v^{\LL(s_1)}} \gamma_{c-1} + \frac{v^{\LL(s_0)}-v^{-\LL(s_0)}}{v^{\LL(s_1)}-v^{-\LL(s_1)}}\gamma_c \Big)\\
&= - \big(v^{\LL(s_1)}-v^{-\LL(s_1)} \big) v^{-c \LL(s_1)} v^{\LL(s_1)} [c]_{v^{\LL(s_1)}} \gamma_{c-1} 
- v^{-c \LL(s_1)} \big({v^{\LL(s_0)}-v^{-\LL(s_0)}} \big)\gamma_c,
\end{align*}
 where
\[
[c]_{v^{\LL(s_1)}} = \frac{v^{c \LL(s_1)}-v^{-c\LL(s_1)} }{v^{\LL(s_1)}-v^{-\LL(s_1)}} \in \Z[v,v^{-1}].
\]
Hence the proposition follows by induction on $c$.
\end{proof}

Following \cite[Theorem~4.26]{BW13} (or \cite{BW16} for more general quantum symmetric pairs with parameters), 
we obtain the $\imath$-canonical bases on 
finite-dimensional simple $\U$-modules and their tensor products. Let us just formulate a special case which we need later on
in the our general weight function $\LL$ setting. 

For $f\in I^m$, define a weight $\texttt{wt}(f) =\sum_{1\le i \le m} \varepsilon_{_{f(i)}} \in \Lambda$.
Let $\theta$ be the involution of the weight lattice $\Lambda$ such that 
\[
\theta(\varepsilon_{i-\hf}) = - \varepsilon_{-i+\hf}, \quad \text{ for all } i \in \I. 
\]
We say two weights $\la, \mu \in \Lambda$ have {\em identical $\imath$-weight} (and denote $\la \equiv_\imath \mu$)
if $\la-\mu$ is fixed by $\theta$.
Define a partial ordering $\preceq$ on the set $I^m$ as follows (cf. \cite[proof of Theorem~5.8]{BW13}):  for $g,f \in I^m$, we let
\begin{equation}  \label{eq:order}
g \preceq f \quad \Leftrightarrow \quad
 \texttt{wt}(g) \equiv_\imath \texttt{wt}(f)  \;
\text{ and } \;  \texttt{wt}(f) - \texttt{wt}(g) \in \N \Pi.
\end{equation} 
We say $g \prec f$ if $g \preceq f$ and $g \neq f$.

\begin{prop}\label{prop:iCBL}
The $\Ui_\LL$-module $\VV_\LL^{\otimes m}$ admits a unique basis $\{b^\imath_f~|~f\in I^m \}$ such that $b^\imath_f$ is $\psi_\imath$-invariant and $b^\imath_f \in M_f +\sum_{g\prec f} v\Z[v] M_g$. 
\end{prop}

\begin{proof}
This is a straightforward $\LL$-generalization of \cite{BW13}, and let us outline the proof for the convenience of the reader. 
By Proposition~\ref{prop:integralform}, $\Upsilon$ and hence also the bar involution $\psi_\imath$ \eqref{eq:ipsi} preserve the integral form of $\VV_\LL^{\otimes m}$, i.e. the $\Z[v,v^{-1}]$-span of $\{M_f \vert f \in I^m\}$. The existence of a $\psi_\imath$-invariant basis $\{b^\imath_f~|~f\in I^m \}$ in the $\Z[v]$-span of Lusztig's canonical basis on $\VV_\LL^{\otimes m}$ follows by applying \cite[Lemma~24.2.1]{L94} (as we showed in \cite[Theorem~4.25]{BW13} in a general based $\U$-module setting). The partial order as stated in the proposition follows from arguments in \cite[proof of Theorem~5.8]{BW13}. 
\end{proof}

\begin{definition}
We call the basis $\{b^\imath_f~|~f\in I^m \}$ constructed in Proposition~\ref{prop:iCBL}  the $\imath$-canonical basis  of $\VV_\LL^{\otimes m}$.
\end{definition}

\begin{rem}  \label{rem:1} 
For each $f\in I^m$, $b^\imath_f$ is the unique element in $\VV_\LL^{\otimes m}$ which is $\psi_\imath$-invariant 
such that $b^\imath_f \in M_f +\sum_{g} v\Z[v] M_g$ (without the partial ordering condition on $g$).
\end{rem}

\subsection{The $\imath$-canonical bases and $\LL$-bases}
  \label{sec:iCB=LB}
  
The double centralizer property in Theorem~\ref{thm:UiHB} specializes to the following double centralizing actions:
\begin{equation*}
\Ui_\LL \; \stackrel{\Phi}{\curvearrowright} \;
\VV^{\otimes m}_\LL \; \stackrel{\Psi}{\curvearrowleft}
\; \HBm^\LL.
\end{equation*}
The following is an $\LL$-variant of \cite[Theorem 5.8]{BW13} (where $\LL(s_0)=\LL(s_1)=1$).
The original proof works here, which uses Lemma~\ref{lem:H0}. 

\begin{prop}   \label{prop:samebar}
The anti-linear bar involution $\psi_\imath: \VV_\LL^{\otimes m} \rightarrow \VV_\LL^{\otimes m}$  is compatible 
with both the bar involution of $\mc{H}^\LL_{B_m}$ and the bar involution of $\Ui_\LL$; 
that is, for all $u \in \VV_\LL^{\otimes m}$,  $h \in \HBm$, and $x \in \Ui_\LL$, we have 
$
\psi_\imath(x u h) =\psi_\imath(x) \, \psi_\imath(u) \overline{h}. \nonumber
$
\end{prop}

Let 
\begin{equation}  \label{eq:Im}
\Ii^m_+ := \{ f \in \Ii^m \mid 0 \leq f(1) \leq f(2) \leq \cdots \leq f(m) \}. 
\end{equation}
Then, as a right $\HBm^\LL$-module, $\VV_\LL^{\otimes m}$ is decomposed as
\begin{align}  \label{eq:VV}
	\begin{split}
\VV_\LL^{\otimes m} 
= \bigoplus_{f \in \Ii^m_+} \Big( 
&\bigoplus_{w \in \, {}^{^{J(f)}}  \WBm} \Q(v) M_{f \cdot w} \Big),
\\
\omega_f: \bigoplus_{w \in \, {}^{^{J(f)}}  \WBm} & \Q(v) M_{f \cdot w} 
\stackrel{\simeq}{\longrightarrow}   \eta_{_{J(f)}} \HBm^\LL, 
\quad M_{f} \mapsto \eta_{_{J(f)}},
	\end{split}
\end{align}
where 
\[J(f) =\{ j~|~0\le j \le m-1,\; f \cdot s_j = f \}.
\]  
It follows by Lemma~\ref{lem:parabolic} and the Hecke algebra action \eqref{Haction} that
\begin{equation}  \label{eq:omega}
\omega_f (M_{f\cdot w}) = \eta_{_{J(f)}} H_w, \qquad \text{ for } f \in \Ii^m_+, ~ w \in {}^{^{J(f)}}  \WBm. 
\end{equation}
By Lemma~\ref{lem:parabolic}(3) each $\eta_{_{J(f)}} \HBm^\LL$ is preserved by the involution $\bar{\ }$. 
Thanks to Proposition~\ref{KLB} and the identification \eqref{eq:VV}, the space $\VV^{\otimes m}_{\LL}$ admits an  $\LL$-basis
\[
\Big\{c_{f\cdot w} := \omega_f^{-1}(C_w^{J(f)}) ~\big |~f \in \Ii^m_+, ~ w \in  {}^{^{J(f)}}  \WBm \Big\}.
\]
We have the following main theorem of this section.

\begin{thm}
  \label{thm:iCB=LB}
The $\imath$-canonical basis and the $\LL$-basis on $\VV^{\otimes m}_\LL$ coincide.
\end{thm}

\begin{proof}
By Proposition~\ref{KLB} and \eqref{eq:omega}, we have 
$c_{f\cdot w}  \in M_{f\cdot w} +\sum_{\sigma \in {}^{^{J(f)}}  \WBm} v \Z[v] M_{f\cdot \sigma}$,
for $f \in \Ii^m_+, ~ w \in  {}^{^{J(f)}}  \WBm$.
By Proposition~\ref{prop:samebar}, $b_{f\cdot w}^\imath $ is $\psi_\imath$-invariant.
By  the existence of $\imath$-canonical basis in Proposition~\ref{prop:iCBL} and 
the uniqueness in Remark~\ref{rem:1}, we must have
$c_{f\cdot w} = b_{f\cdot w}^\imath$. The theorem is proved. 
\end{proof}

\begin{example}\label{ex:iCBVV}
The $\imath$-canonical basis on $\VV_\LL$ is given as follows: for $i \in \Ii$ and $i >0$, 
\begin{align*}
&u_i, \quad u_{i \cdot s_0}, &\quad \text{ if } \LL(s_0)=0; \\
&u_i, \quad u_{i \cdot s_0} + v^{\LL(s_0)} u_i, &\quad \text{ if } \LL(s_0)> 0;\\
&u_i, \quad u_{i \cdot s_0} - v^{-\LL(s_0)} u_i, &\quad \text{ if } \LL(s_0)< 0.
\end{align*}
The above $\imath$-canonical basis on $\VV$ coincides with Lusztig's example for $\LL$-basis in \cite[\S5.5]{L03}.
\end{example}

\section{The $\jmath$-Schur duality with 2 parameters}
 \label{sec:Uj}

In this section, we establish the duality between a coideal subalgebra $\Uj$ of the quantum group for $\mathfrak{sl}_{2r+1}$ 
and the Hecke algebra of type $B$. This section is parallel to Sections~\ref{sec:Ui}--\ref{sec:CBLB}, 
and we shall omit many redundant details to avoid much repetition. 
We sometimes use the same notation in similar circumstances, as both cases are special cases for $\mathfrak{sl}_k$,
with $k={2r+1}$ here (and $k=2r+2$ in  Sections~\ref{sec:Ui}--\ref{sec:CBLB}).

\subsection{The quantum symmetric pair $(\U,\Uj)$}
Let $r$ be a positive integer. We set 
\begin{equation}  \label{eq:Ij}
\I  = \Big\{-r+\hf, -r+\frac{3}{2}, \dots, r- \hf \Big\},
\qquad
\I^\jmath = \Big\{\hf, \frac32, \dots, r-\hf \Big\}. 
\end{equation}
Let 
$
\Pi= \big \{\alpha_i=\varepsilon_{i-\hf}-\varepsilon_{i+\hf} \mid i \in \I \big \}
$
be the simple system of type $A_{2r}$, and $\Phi$ the associated root system. Denote the weight lattice by 
$$
\Lambda = \bigoplus_{i=-r}^r \Z \varepsilon_i.
$$  

Let 
\[
\U = \U_q(\mathfrak{sl}_{2r+1}) \otimes_{\Q(q)} \Q(p,q)
\] 
be the quantum group of type $A_{2r}$ over $\Q(p,q)$ with the standard generators $E_{i}$, $F_{i}$ and $K^{\pm 1}_{i}$ for $i \in \I$. 
We denote by $\psi: \U \to \U$ and $\Delta: \U \to \U \otimes \U$ be the bar involution and comultiplication on $\U$ 
given by the same formulas as in \eqref{eq:psi} and \eqref{eq:Delta}.

 Let $(\U,\Uj)$ be the quantum symmetric pair (cf. \cite{Le99}) over $\Q(p,q)$ with the following Satake diagram:
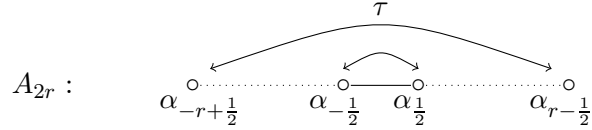
\begin{center}
\begin{tikzpicture}
\draw (-1.5,0) node {$A_{2r}:$};
 \draw[dotted]  (0.6,0) node[below]  {$\alpha_{-r+\hf}$} -- (2.4,0) node[below]  {$\alpha_{-\hf}$} ;
 \draw (2.6,0)
 -- (3.4,0) node[below]  {$\alpha_{\hf}$};
 \draw[dotted] (3.6,0) -- (5.4,0) node[below] {$\alpha_{r-\hf}$} ;
\draw (0.5,0) node (-r) {$\circ$};
 \draw (2.5,0) node (-1) {$\circ$};
\draw (3.5,0) node (1) {$\circ$}; 
\draw (5.5,0) node (r) {$\circ$};
\draw[<->] (-r.north east) .. controls (3,1) .. node[above] {$\tau$} (r.north west) ;
\draw[<->] (-1.north) .. controls (3,0.5) ..  (1.north) ;
\end{tikzpicture}
\end{center}

The $\Q(p,q)$-algebra $\Uj$ is the $\Q(p,q)$-subalgebra of $\U$ generated by (for $i \in \I^\jmath$)
 \begin{align}
  \label{eq:Uj}
  \begin{split}
\bk_{i} = K_{i}K^{-1}_{{-i}}, \quad 
&\be_{i} =  E_{i} + F_{{-i}}K^{-1}_{i}\,\, (i \neq \hf), \quad
\bff_{i} = E_{{-i}} + K^{-1}_{{-i}}  F_{i} \,\, (i \neq \hf),  \\
& \be_{{\hf}} =  E_{{\hf}} + p^{-1} F_{{-{\hf}}}K^{-1}_{{\hf}}, \quad
\bff_{{\hf}} = E_{{-{\hf}}} + p K^{-1}_{{-{\hf}}}  F_{{\hf}}.  
\end{split}
\end{align}

The $\Q(p,q)$-algebra $\Uj$ has the following presentation: it is generated by
$\be_i, \bff_i, \bk_i^{\pm 1}$ (for $i \in \I^\jmath$), 
subject to the following relations (for $i,j \in \I^\jmath$):
\begin{align*}
 \bk_{i} \bk_{i}^{-1} &= \bk_{i}^{-1} \bk_{i} =1,
 \quad
  \bk_{i} \bk_{ j} = \bk_{ j} \bk_{i}, \displaybreak[0]\\
 \bk_{i} \be_{ j} \bk_{i}^{-1} &= q^{(\alpha_i-\alpha_{-i}, \alpha_j)} \be_{j}, 
 \quad
  \bk_{i} \bff_{ j} \bk_{i}^{-1} = q^{-(\alpha_i-\alpha_{-i}, \alpha_j)} \bff_{j}, \displaybreak[0]\\
 \be_{i} \bff_{ j} -\bff_{ j} \be_{i} &= \delta_{i,j} \frac{\bk_{i}
 -\bk^{-1}_{i}}{q-q^{-1}}, \displaybreak[0]\\
 \be_{i} \be_{j} &= \be_{j} \be_{i}, \quad
  \bff_{i} \bff_{ j}  = \bff_{ j}  \bff_{i} \qquad  &\text{if } |i-j|>1, \displaybreak[0]\\
 \be_{i}^2 \be_{ j} +\be_{ j} \be_{i}^2 &= (q+q^{-1}) \be_{i} \be_{ j} \be_{i}, 
 \quad
  \bff_{i}^2 \bff_{j} +\bff_{ j} \bff_{i}^2 = (q+q^{-1}) \bff_{i} \bff_{ j} \bff_{i},
   & \text{if } |i-j|=1,\displaybreak[0]\\
\be_{ \hf}^2 \bff_{ \hf} + \bff_{ \hf} \be_{ \hf}^2 &= (q + q\inv)
  \Big(\be_{ \hf} \bff_{ \hf} \be_{\hf} - \be_{ \hf}\big(pq\bk_{ \hf} + p\inv q\inv \bk_{ \hf}\inv\big) \Big), \nonumber\\
\bff_{ \hf}^2 \be_{ \hf} + \be_{ \hf} \bff_{ \hf}^2 &= (q + q\inv)
  \Big(\bff_{ \hf} \be_{ \hf} \bff_{ \hf} - \big(pq\bk_\hf + p\inv q\inv \bk_\hf\inv \big)\bff_{\hf} \Big). \nonumber
\end{align*}
In contrast to the $\Ui$ case, the presentation of $\Uj$ depends on the parameter $p$.

The following counterpart of Lemma~\ref{lem:bar} follows from the above presentation.
\begin{lem}
There exists a unique $\Q$-algebra bar involution $\jbar$ on the algebra $\Uj$ such that $p \mapsto p\inv$, $q \mapsto q\inv$, $\bk_{i} \mapsto \bk_{i}\inv$, $\be_{i} \mapsto \be_{i}$, and $\bff_{i} \mapsto \bff_{i}$, for $i \in \I^\jmath$.
\end{lem}

Just as Proposition~\ref{prop:intetwiner} for $\Ui$, 
we have the intertwiner $\Upsilon \in \widehat{\U}$ between the involution $\psi$ on $\U$ and the involution  $\psi_\jmath$ on $\Uj$ such that
\[
\psi_\jmath(u) \cdot \Upsilon = \Upsilon \cdot \psi({u}), \qquad \forall u \in \Uj.
\]
This is a straightforward multiparameter variant of \cite[Theorem~6.4]{BW13} (cf. \cite{BK15, BW16}). 


Consider a $\Q(p,q)$-valued function $\zeta$ on $\Lambda$ such that ($\forall \mu \in \Lambda, \; i \in \Big\{ \hf, \ldots, r-\hf \Big \}$)
\begin{align}
  \label{eq:zeta2}
\begin{split}
\zeta (\mu+\alpha_i) &= -q^{(\alpha_i -\alpha_{-i}, \mu+\alpha_i)} \zeta (\mu),  \\
\zeta (\mu+\alpha_{-i}) &= -q^{(\alpha_{-i}, \mu+\alpha_{-i}) - (\alpha_{i}, \mu)} \zeta (\mu), \\
\zeta(\mu + \alpha_\hf) &= -p q^{(\alpha_\hf - \alpha_{-\hf},\mu + \alpha_\hf) - 1} \zeta(\mu), \\
\zeta(\mu + \alpha_{-\hf}) &= -p\inv q^{(\alpha_{-\hf}, \mu + \alpha_{-\hf}) - (\alpha_\hf, \mu) + 1} \zeta(\mu). 
\end{split}
\end{align}
Such $\zeta$ clearly exists. For any weight $\U$-module $M$, 
we obtain a $\Q(p,q)$-linear map $\widetilde{\zeta}: M \rightarrow M$ as in \eqref{eq:zetatd}.  
Let $w_0$ be the longest element of the Weyl group $W_{A_{2r}}$ and $T_{w_0}$ the associated braid group element.
The following multiparameter variant of \cite[Theorem~6.6]{BW13} holds by the same proof. 
\begin{prop}
For any finite-dimensional $\U$-module $M$, the composition map
$
\mc{T} := \Upsilon\circ \widetilde{\zeta} \circ T_{w_0}: M \longrightarrow M
$
is a $\Uj$-module isomorphism.
\end{prop}

\subsection{The $(\Uj,\HBm)$-duality}

Let $\Ij = \I \pm \hf$. Let $\VV = \bigoplus_{a \in \Ij} \Q(p,q)u_a$ be the natural representation of $\U$. 
The $\U$-module structure of $\VV$ can be visualized as follows:
\begin{align}
\xymatrix{
u_{r} \ar@/_/[r]_{E_{ {r-\frac{1}{2}}}} & 
u_{r-1} \ar@/_/[l]_{F_{ {r-\frac{1}{2}}}} \ar@/_/[r]_{E_{ {r-\frac{3}{2}}}} & 
\cdots \ar@/_/[l]_{F_{ {r-\frac{3}{2}}}} \ar@/_/[r]_{E_{ {\frac{1}{2}}}} & 
u_{0} \ar@/_/[l]_{F_{ {\frac{1}{2}}}} \ar@/_/[r]_{E_{ {-\frac{1}{2}}}} & 
\cdots \ar@/_/[l]_{F_{ {-\frac{1}{2}}}} \ar@/_/[r]_{E_{ {-r+\frac{1}{2}}}} & 
u_{-r} \ar@/_/[l]_{F_{ {-r+\frac{1}{2}}}}}. \nonumber
\end{align}
We regard the $\U$-module $\VV^{\otimes m}$ as a $\Uj$-module by restriction.

Recall from \eqref{Mf} the element $M_f \in \VV^{\otimes m}$, for any $f \in \Ij^m$ (except that $I$ here is understood as in \eqref{eq:Ij}). 
The Weyl group $W_{B_m}$ acts on $\Ij^m$ in the obvious way. The Hecke algebra $\HBm$ acts on $\VV^{\otimes m}$ as follows:
\begin{align}   \label{Haction2}
\begin{split}
 M_f \cdot H_i &=
 \begin{cases}
 q^{-1}M_f, & \text{ if }  f(i) = f(i+1);\\
 M_{f \cdot  s_i}, & \text{ if }  f(i) < f(i+1);\\
 M_{f \cdot  s_i} + (q^{-1} - q) M_{f}, & \text{ if } f(i) > f(i+1);
 \end{cases}\\
 M_f \cdot H_0 &= 
 \begin{cases} 
  p^{-1} M_{f}, &\text{ if } f(1) = 0;\\
 M_{f \cdot s_0}, &\text{ if } f(1) > 0;\\
  M_{f \cdot  s_0} + (p^{-1} - p) M_{f}, & \text{ if } f(1) <0.
 \end{cases}
 \end{split}
\end{align}
Summarizing, we shall depict the actions of $\Uj$ and $\HBm$ on $\VV^{\otimes m}$ as
\begin{equation}  \label{LRact2}
\Uj \; \stackrel{\Phi}{\curvearrowright} \;
\VV^{\otimes m} \; \stackrel{\Psi}{\curvearrowleft}
\; \HBm.
\end{equation}

We fix $\zeta$ in \eqref{eq:zeta2} such that $\zeta(\epsilon_{-r}) = 1$. Then, we have
\begin{align}
\zeta(\epsilon_{-i}) = \begin{cases}
(-q)^{-r+i} \qu \text{if } i \neq 0,\\
 (-q)^{r}  p\qu \text{if } i = 0,
\end{cases} \nonumber
\end{align}
for all $i \in \{ -r, - r+1, \ldots, r \}$. 

\begin{lem}
 \label{lem:commute2}
The actions of $H_0$ and $\mc{T}^{-1}$on $\VV$ coincide.
\end{lem}

\begin{proof}
We define 
\begin{align*}
\VV^+ &=\bigoplus_{j\in \I^\jmath}\Q(p,q)( u_{-j-\hf}+pu_{j+\hf}) \bigoplus \Q(p,q) u_0,\\
\VV^- &= \bigoplus_{j\in \I^\jmath}\Q(p,q)( u_{-j-\hf}-p^{-1}u_{j+\hf}).
\end{align*}
 By direct calculations, we have, for $j \in \I^\jmath$, 
\begin{align}
\bff_{\alpha_{\frac{1}{2}}} \cdot u_0 &= u_{-1} + p u_1, \nonumber\\
\bff_{i} \cdot \big(u_{-j-\frac{1}{2}} + pu_{j+\frac{1}{2}}\big) 
&= \delta_{i, j+1} \cdot \big(u_{-(j+1)-\frac{1}{2}} + p u_{(j+1) + \frac{1}{2}}\big),\nonumber\\
%
%
\be_{i} \cdot \big(u_{-j-\frac{1}{2}} + pu_{j+\frac{1}{2}}\big) 
&= \delta_{i, j} \cdot \big(p^{-\delta_{\hf,i}}u_{-(j-1)-\frac{1}{2}} + p u_{(j-1) + \frac{1}{2}}\big), \nonumber\\
\bff_{i} \cdot \big(u_{-j-\frac{1}{2}} - p\inv u_{j+\frac{1}{2}}\big) 
&= \delta_{i, j+1} \cdot \big(u_{-(j+1)-\frac{1}{2}} - p\inv u_{(j+1) + \frac{1}{2}}\big), \nonumber\\
\be_{i} \cdot \big(u_{-j-\frac{1}{2}} - p\inv u_{j+\frac{1}{2}}\big) 
&= \delta_{i, j} \cdot \big(p^{-\delta_{\hf,i}}u_{-(j-1)-\frac{1}{2}} - p\inv u_{(j-1) + \frac{1}{2}}\big). \nonumber
\end{align}
Hence, $\VV =\VV^+ \oplus \VV^-$ as a $\Uj$-module. 
Furthermore,  $H_0$ acts as the scalar multiplication by $p^{-1}$ (resp., $-p$) on $\VV^+$ (resp., $\VV^-$). 

Since we have
$
T_{w_0}(u_{j}) = (-q)^{r-j} \cdot u_{-j},  
$
we obtain
\begin{align}
\widetilde{\zeta} \circ T_{w_0}(u_{j}) = \begin{cases}
u_{-j} \qu \text{if } j \neq 0,\\
p \cdot u_0 \qu \text{if } j = 0.
\end{cases} \nonumber
\end{align}
On the other hand, one computes the first term of $\Upsilon$
as $\Upsilon_{\alpha_\hf + \alpha_{-\hf}} = (p - p\inv)F_{\alpha_\hf} F_{\alpha_{-\hf}}$. Hence, we have
$
\mc{T}(u_0) = p u_0,  \mc{T}(u_1) = u_{-1} - (p\inv - p)u_1$, and $ \mc{T}(u_{-1}) = u_1, 
$ 
which imply that the action of $\mc{T}^{-1}$ on $\VV^+$ and $\VV^{-}$ are given by scalar multiplication 
by $p^{-1}$ and $-p$, respectively. The lemma follows.
\end{proof}

Now with the help of Lemma~\ref{lem:commute2}, we obtain the 
following counterpart of Theorem~\ref{thm:UiHB} by the same argument.
\begin{thm} [$\jmath$-Schur duality] 
The actions of $\Uj$ and $\HBm$ on $\VV^{\otimes m}$ \eqref{LRact2} commute and form double centralizers; that is, 
$$
\Phi(\Uj)=
\End_{\HBm}(\VV^{\otimes m}), \quad
\End_{\Uj}(\VV^{\otimes m})^{\operatorname{op}} =
\Psi(\HBm).
$$
\end{thm}

\subsection{The $\jmath$-canonical basis and $\LL$-basis}
 \label{sec:L}

 All the results in \S\ref{sec:iCB}--\ref{sec:iCB=LB} admit natural counterparts in the setting of $\Uj$. 
 The proofs are similar or easier in the $\Uj$ setting (e.g., the integrality of the intertwiner $\Upsilon$ is completely the same as in \cite{BW13}.)
 So we shall be brief.

Given a weight function $\LL: W_{B_m} \rightarrow \Z$, by a base change we have a $\Q(v)$-algebra 
\begin{equation*}
\Uj_\LL = \Uj \otimes_{\Q(p,q)} \Q(v).
\end{equation*}
Recall the $\U_\LL$-module $\VV^{\otimes m}$ admits a bar involution $\psi$ using the quasi-$R$-matrix $\Theta$ ({\cite[Chap.~ 4]{L94}}).
We define another anti-linear bar involution on the $\Uj_\LL$-module $\VV_\LL^{\otimes m}$ as 
\[
\psi_\jmath = \Upsilon \circ \psi.
\] 
Entirely similar to \cite{BW13}, we can establish the $\jmath$-canonical bases on  
finite-dimensional simple $\U$-modules and their tensor products.
In particular $\VV_\LL^{\otimes m}$ admits a $\jmath$-canonical basis (similar to Proposition~\ref{prop:iCBL}).
As in Proposition~\ref{prop:samebar}, we have compatible bar maps in the following sense:
for all $u \in \VV_\LL^{\otimes m}$,  $h \in \HBm^\LL$, and $x \in \Uj_\LL$, we have 
\begin{equation}  
\psi_\jmath(x u h) =\psi_\jmath(x) \, \psi_\jmath(u) \overline{h}. \nonumber
\end{equation}

We still define 
$\Ij^m_+$ as in \eqref{eq:Im}, and the decomposition of $\VV_\LL^{\otimes m}$
as a right $\HBm^\LL$-module as in \eqref{eq:VV}. Then $\VV_\LL^{\otimes m}$ admits a bar involution and 
an $\LL$-basis (inherited from $\HBm^\LL$).
Keep in mind again $\Ij^m_+$ and $\VV_\LL^{\otimes m}$ are slightly different from those in \S\ref{sec:iCB=LB}, because $I$ here is understood as in \eqref{eq:Ij}.
We have the following counterpart of Theorem~\ref{thm:iCB=LB}. 

\begin{thm}
The $\jmath$-canonical basis of $\VV_\LL^{\otimes m}$ is identical to the $\LL$-basis of $\VV_\LL^{\otimes m}$.
\end{thm}

\begin{example}\label{ex:j:iCBVV}
We have the following  $\jmath$-canonical basis for $\VV_\LL$ (for $1\le i \le r$): 
\begin{align*}
u_0,\quad u_i, \quad u_{i \cdot s_0}, &\quad \text{ for } \LL(s_0)=0; \\
u_0,\quad u_i, \quad u_{i \cdot s_0} + v^{\LL(s_0)} u_i, &\quad \text{ for } \LL(s_0)> 0;\\
u_0, \quad u_i, \quad u_{i \cdot s_0} - v^{-\LL(s_0)} u_i &\quad \text{ for } \LL(s_0)< 0.
\end{align*}
Again this example coincides with Lusztig's example in \cite[\S5.5]{L03}.
\end{example}


\end{document}